\documentclass{amsart}

\newtheorem{theorem}{Theorem}[section]
\newtheorem{lemma}[theorem]{Lemma}

\theoremstyle{remark}
\newtheorem{remark}{Remark}[section]

\begin{document}

\title{Knots in homology lens spaces determined by their complements}

\author{Kazuhiro Ichihara}
\address{Department of Mathematics, College of Humanities and Sciences, Nihon University, 3-25-40 Sakurajosui, Setagaya-ku, Tokyo 156-8550, JAPAN}
\email{ichihara.kazuhiro@nihon-u.ac.jp}

\author{Toshio Saito}
\address{Department of Mathematics, Joetsu University of Education, 1 Yamayashiki, Joetsu 943-8512, JAPAN}
\email{toshio@juen.ac.jp}
\date{\today}



\begin{abstract}
In this paper, we consider the knot complement problem for not null-homologous knots in homology lens spaces. 
Let $M$ be a homology lens space with $H_1(M; \mathbb{Z}) \cong \mathbb{Z}_p$ and $K$ a not null-homologous knot in $M$. 
We show that, $K$ is determined by its complement if $M$ is non-hyperbolic, $K$ is hyperbolic, and $p$ is a prime more than 7, or, if $M$ is actually a lens space $L(p,q)$ and $K$ represents a generator of $H_1(L(p,q))$. 
\end{abstract}

\maketitle

\section{Introduction}
If a pair of knots in a 3-manifold is equivalent, meaning that a homeomorphism of the ambient manifold takes one knot onto the other, then the complements of the two knots are homeomorphic. 
In general, its converse may not hold. 
In view of this, a knot $K$ in a 3-manifold $M$ is said to be \textit{determined by its complement} if the complement $M - K$ is not orientation-preservingly homeomorphic to the complement of any knot in $M$ unless it is orientation-preservingly equivalent to $K$. 

For the case of knots in the 3-sphere $S^3$, it had been conjectured that any non-trivial knot in $S^3$ is determined by its complement. 
As a question, it was stated by Tietze \cite[p.83]{Tietze} in 1908. 
About 80 years later, in 1989, Gordon and Luecke \cite{GordonLuecke} proved the conjecture affirmatively: 
If two knots in $S^3$ have homeomorphic complements, then they are equivalent (\cite[Theorem 1]{GordonLuecke}). 

Then, the question was generalized to the knots in general 3-manifolds. 
The following is now called the \textit{Oriented Knot Complement Conjecture}, stated in 
\cite[Conjecture 6.2.]{Gordon} and \cite[Problem 1.81(D)]{Kirby}: 
Any knot in a closed, oriented 3-manifold with the complement which is not a solid torus is determined by its complement. 
Throughout the sequel, for simplicity, we will always assume that no knots have solid torus exteriors. 

Note that the homeomorphism between the knot complements is assumed to preserve the orientations.
Otherwise the situation becomes more complicated. See the remark in the end of this section. 

The conjecture is verified to be true for knots with Seifert fibered exteriors in Seifert fibered spaces by Rong in \cite{Rong} and for non-hyperbolic knots in lens spaces, and moreover, in closed, atoroidal and irreducible Seifert fibered spaces by Matignon in \cite{Matignon2010}. 
In particular, all the non-hyperbolic knots except for the core knots (the cores of Heegaard solid tori) in lens spaces are determined by their complements. 
However it is still widely open in general. 

In these decades, the conjecture is mostly addressed by using Heegaard Floer homology developed by Ozsv\'{a}th and Szab\'{o}. 
In fact, an alternative proof the theorem of Gordon and Luecke (\cite[Theorem 1]{GordonLuecke}) can be given by using Heegaard Floer homology. 
See \cite[Section 4]{OzsvathSzabo} for example. 

Furthermore, the following result was obtained by Gainullin \cite{Gainullin}: 
Null-homologous knots in L-spaces are determined by their complements.
In particular, also are null-homologous knots in lens spaces (\cite[Theorem 8.2]{Gainullin}). 
Together with \cite[Corollary 8.3]{Gainullin}, which is also previously shown in \cite{Christensen} independently, 
it implies that if $p$ is square-free, then knots in a lens space $L(p,q)$ except for the core knots are determined by their complements.  

In this paper, as a generalization of the result above, we consider a \textit{homology lens space}, defined as a closed 3-manifold with homology groups isomorphic to those of a lens space, and knots not null-homologous in the ambient manifolds. 
(Here and throughout the rest of the paper, all homology is considered with $\mathbb{Z}$-coefficients.)
By refining the argument of \cite[Corollary 8.3]{Gainullin} and \cite{Christensen}, we show the following. 

\begin{theorem}\label{MainThm}
Let $M$ be a homology lens space with $H_1(M)  \cong \mathbb{Z}_p$ for a prime $p$, and $K$ a hyperbolic knot not null-homologous in $M$. 
If $M$ is non-hyperbolic and $p>8$, then $K$ is determined by its complement. 
\end{theorem}

Due to the result of Gainullin \cite{Gainullin}, in addition to the above setting, if $M$ is an L-space, then the assumption of not null-homologousness can be removed. 
For example, for a prime $p>8$, if $M$ is a connected sum of a lens space $L(p,q)$ and the Poincar\'{e} homology sphere, then all the hyperbolic knots in $M$ are determined by their complements. 

When we restrict our attention to the knots, not only not null-homologous, but representing generators of the homology groups, we have the following. 

\begin{theorem}\label{MainThm2}
Let $M$ be a homology lens space with $H_1(M)  \cong \mathbb{Z}_p$ for a positive integer $p$, and $K$ a knot in $M$ representing a generator of $H_1(M)$. 
\begin{enumerate}
    \item If $M$ is a lens space, then $K$ is determined by its complement. 
    \item If $M$ is a spherical manifold with $p$ containing a prime factor at least 7 and $K$ is hyperbolic, then $K$ is determined by its complement. 
    \item If $M$ is non-hyperbolic with $p$ containing a prime factor at least 11  and $K$ is hyperbolic, then $K$ is determined by its complement. 
\end{enumerate}
\end{theorem}

Actually, the theorems above are obtained as corollaries of the next theorem. 

\begin{theorem}\label{Temp}
Let $M$ be a homology lens space with $H_1(M) \cong \mathbb{Z}_p$ for a positive integer $p$. 
Let $K$, $K'$ be not null-homologous knots in $M$ with meridians $m$, $m'$ and the exteriors $E(K;M)$, $E(K';M)$ respectively. 
Suppose that there exists a homeomorphism $h$ from $E(K;M)$ to $E(K';M)$, and let $\Delta$ be the distance of the slopes on $\partial E(K;M)$ represented by $m$ and $h^{-1}(m')$. 
(i) If $p$ is prime, then  $p\,|\,\Delta$ holds. 
(ii) If $K$ represents a generator of $H_1(M)$, then $p_0 \,|\,\Delta$ holds for any prime factor $p_0$ of $p$. 
\end{theorem}

Very recently, for lens spaces, independently from our work, Ito obtained a further extension of the result of Gainullin \cite{Gainullin} (in private communication).
By applying our methods, we have a simple alternative proof of the following, which is a part of his results. 

\begin{theorem}\label{L4q}
In a lens space $L(4,q)$ for an odd integer $q$, all the knots are determined by their complements. 
\end{theorem}

\begin{remark}
In \cite{Mathieu90} and \cite{Mathieu92}, Mathieu gave the first examples of pairs of knots in some Seifert fibered spaces with orientation-reversing homeomorphic (Seifert fibered) complements. 
Such examples were extended and classified by Rong in \cite{Rong}, and in \cite{Matignon2010}, Matignon gave a complete classification of non-hyperbolic knots in lens spaces not determined by unoriented knot complements. 
In \cite{BleilerHodgsonWeeks}, Bleiler, Hodgson and Weeks presented the first such hyperbolic example in a lens space. 
Recently, by generalizing their example, the first example of such hyperbolic knots in a hyperbolic 3-manifold was given by the first author and Jong in \cite{IchiharaJong}. 
\end{remark}

\section{Proofs}

In this section, we give proofs of Theorems~\ref{MainThm}, \ref{MainThm2}, and \ref{Temp} given in Section 1. 
Before starting proofs, we prepare some terminologies. 

The following operation is called \textit{Dehn surgery on a knot $K$ along slope $\gamma$}. 
Given a knot $K$ in a closed oriented 3-manifold $M$ and a slope $\gamma$, construct a new manifold $M_K( \gamma )$ by gluing a solid torus to the knot exterior $E(K;M) = \mathrm{cl}(M - N(K))$, that is, $M_K(\gamma) = E(K;M) \cup ( D^2 \times S^1) $, where the slope $\gamma$ represents the curve identified with $ \partial D^2$. 
If a meridian-longitude coordinates is fixed on $\partial N(K)$, then such a slope is parametrized by $\mathbb{Q} \cup \{ 1/0 \}$, where the meridian of $K$ corresponds to $1/0$. 
In that case, if the slope $\gamma$ is identified with $p/q$, then $M_K( \gamma )$ is also denoted by $M_K( p/q )$, and the Dehn surgery is called the $(p/q)$-surgery. 

We first prepare the following lemma. 

\begin{lemma}\label{lem}
Let $M$ be a homology lens space with $H_1(M) \cong \mathbb{Z}_p$ for a positive integer $p$. 
Let $K$ be a knot not null-homologous in $M$. 
Then, there exist a non-zero integer $w$ not a multiple of $p$ and an integer $q$ coprime to $p$ such that, with a suitable longitude of $K$, the manifold $M_K(n/n')$ obtained by the $(n/n')$-surgery on $K$ has the first homology group $H_1(M_K(n/n'))$ presented as follows. 
\[
H_1(M_K(n/n')) = \langle [m], [\mu] \ | \ n[m] + n'w [\mu] =0, \ - qw[m] + p [\mu] =0 \rangle.
\]
Moreover, if $H_1(M_K(n/n')) \cong \mathbb{Z}_p$, then there also exist integers $\alpha$ and $\beta$ with $0 \le \alpha,\beta \le p-1$ and $(\alpha, \beta) = 1$ if $\alpha, \beta$ are both non-zero or $(p, \alpha)=1$ (resp. $(p,\beta)=1$) if $\beta =0$ (resp. $\alpha=0$) such that 
\begin{eqnarray}
\label{eq2}
 n \alpha + n'w \beta &\equiv& 0\ \  (\bmod\  p),\ \mathit{and}\\
\label{eq3} 
qw \alpha &\equiv& 0\ \  (\bmod\  p). 
\end{eqnarray}
In particular, if $K$ represents a generator of $H_1(M)$, $w$ can be taken so that 
\begin{eqnarray}
\label{eq4}
(w, p) = 1. 
\end{eqnarray}
\end{lemma}

\begin{proof}
Let $M$ be a homology lens space with $H_1(M) \cong \mathbb{Z}_p$ for a positive integer $p$. 
It is shown in \cite{LuftSjerve} that $M$ is $\mathbb{Z}$-homology equivalent to a lens space $L(p,q)$ for coprime integers $p,q$, i.e., 
there is a map $f:M \to L(p,q)$ which induces $\mathbb{Z}$-homology isomorphisms in all dimensions.
It then follows from \cite[Theorem 3.1]{Guilbault} that there exist a homology sphere $\Sigma$ and a knot $K_0$ in $\Sigma$ on which the $(-p/q)$-surgery yields $M$, that is, $\Sigma_{K_0} (-p/q) \cong M$. 
Here we take the meridian-preferred longitude system $(\mu, \lambda)$ on $\partial E(K_0;\Sigma)$ as follows. 
Let $\mu$ be an oriented meridian of $K_0$, which represents a generator of $H_1(E(K_0;\Sigma)) \cong \mathbb{Z}$. 
Since $\Sigma$ is a homology sphere, $K_0$ bounds a compact, connected, orientable surface, say $F$, in $\Sigma$. 
Such a surface $F$ is called a Seifert surface of $K_0$ and $\lambda=F\cap \partial E(K_0;\Sigma)$ is called the preferred longitude of $K_0$. 
Note that $[\lambda]=0$ in $H_1(E(K_0;\Sigma))$. 
We may suppose that $[\mu]\cdot [\lambda]=1$. 

Let $K$ be a knot which is not null-homologous in $M$. 
We isotope $K$ into $E(K_0;\Sigma)$ and define an oriented meridian $m$ and the preferred longitude $\ell$ of $K$, in the same way as for $K_0$, by regarding $K$ is a knot in $\Sigma$. 
Suppose that $K$ and $K_0$ have the linking number $w$ in $\Sigma$ (which means the algebraic intersection number of $K$ and a Seifert surface $F$ of $K_0$). 
Then $[\lambda] = w [m]$ and $[\ell] = w [\mu]$ holds in $H_1(E(K_0;\Sigma))$. 
In particular, $p [\mu] - q [\lambda] = p [\mu] - qw[m]$ holds in $H_1(E(K_0;\Sigma))$. 
Note that, since we assume that $K$ is not null-homologous in $M$, $p \not| \ w$ holds. 

Now we calculate $H_1 (M_K(n/n')) $ by considering the $(-p/q)$-surgery on $K_0$ and the $(n/n')$-surgery on $K$.
We see that the $(-p/q)$-surgery on $K_0$ induces a relation $p [\mu] - qw[m] = 0$ and 
the $(n/n')$-surgery on $K$ induces a relation $n[m] + n'w [\mu] = 0$ in $H_1 (M_K(n/n')) $. 
Therefore we have the following presentation. 
\[
H_1(M_K(n/n')) = \langle [m], [\mu] \ | \ n[m] + n'w [\mu] =0, \ - qw[m] + p [\mu] =0 \rangle.
\]

We now assume $H_1(M_K(n/n')) \cong \mathbb{Z}_p$. Then, there is an isomorphism $\varphi$ from $H_1(M_K(n/n'))$ to 
$\mathbb{Z}_p= \langle x \ | \ px = 0 \rangle$ such that $\varphi([m])=\alpha x$ and $\varphi([\mu])=\beta x$, 
where $0 \le \alpha,\beta \le p-1$. 
Moreover, if $\alpha, \beta \ne 0$, then $(\alpha, \beta)=1$, and 
if $\beta =0$ (resp. $\alpha=0$), then $(p, \alpha)=1$ (resp. $(p,\beta)=1$) because $\varphi$ is surjective. 
Note that 
\begin{eqnarray*}
\varphi (n[m] + n'w [\mu])  = (n \alpha + n'w \beta) x = 0,\\
\varphi (- qw[m] + p [\mu])  = (- qw \alpha + p \beta) x = 0.  
\end{eqnarray*}
These imply that we have the equations $(1)$ and $(2)$ in Lemma \ref{lem}. 

To obtain the final conclusion, we further assume that $K$ represents a generator of $H_1(M)$. 
Recall that $[K] = [ \ell] = w [\mu]$ in $H_1(E(K_0;\Sigma))$ and $p [\mu ] =0$ in $H_1(M)$. 
Suppose for a contradiction that $(w, p) \ne 1$, and set $c=\gcd(w, p)$. 
Then $w = w' c$ and $p = p' c$ for integers $p'(>1)$ and $w'(>1)$ respectively.   
Then 
\[
p' [K] = p' w [\mu] = p' w' c [\mu] = w' p [\mu] = 0 
\]
in $H_1(M)$. 
Since we suppose $c > 1$ and hence $p' < p$, this contradicts that $K$ represents a generator of $H_1(M)$. 
\end{proof}

\begin{proof}[Proof of Theorem~\ref{Temp}]
Let $M$ be a homology lens space with $H_1(M) \cong \mathbb{Z}_p$ for a positive integer $p$, and let $K$ be a not null-homologous knot in $M$. 
We set a meridian-longitude system $(m,\ell)$ for $K$ as in the proof of Lemma~\ref{lem}. 

Let $K'$ be a not null-homologous knot in $M$ with the exterior $E(K';M)$. 
Suppose that there exists a homeomorphism $h$ from the exterior $E(K;M)$ of $K$ to $E(K';M)$. 
Let $m'$ be a meridian of $K'$, and take the slope represented by $h^{-1}(m')$ on $\partial E(K;M)$, and assume that it corresponds to $n/n' \in \mathbb{Q} \cup \{ 1/0 \}$. 
Then we see that the $(n/n')$-surgery on $K$ in $M$ yields $M$ again. 
That is, $M_K(n/n') \cong M$, in particular, $H_1 (M_K(n/n')) \cong H_1 (M) \cong \mathbb{Z}_p$ holds. 
Since $m$ and $h^{-1}(m')$ corresponds to $1/0$ and $n/n'$ with respect to the meridian-longitude system $(m,\ell)$, the distance $\Delta$ of the slopes represented by $m$ and $h^{-1}(m')$ is equal to $1\cdot n' - n\cdot 0=n'$. 
If $K'$ is equivalent to $K$, then the distance $\Delta=n'$ is 0, and so, $p\, |\, \Delta$ holds. 
Thus, in the following, we suppose that $K'$ is not equivalent to $K$, i.e., $n' \ne 0$. 

Since $M_K(n/n')$ is assumed to be homeomorphic to $M$ and $H_1(M) \cong \mathbb{Z}_p$, we have the equations $(1)$ and $(2)$ in Lemma~\ref{lem}. 
Since we assume that $K$ is not null-homologous in $M$, the following also holds: 
\begin{eqnarray}
\label{eq8}
p \not| \ w. 
\end{eqnarray}

To prove the assertion (i), we assume that $p$ is a prime. 
If $\alpha \ne 0$, then the equation (\ref{eq3}) implies $w \equiv 0\ (\bmod\  p)$ because $p$ and $q$ are relatively prime. 
This contradicts the equation (\ref{eq8}). 
Thus we see that $\alpha = 0$ and hence $\beta \ne 0$. 
Then we have  $n'w \equiv 0\ (\bmod\  p)$ by the equation (\ref{eq2}) and hence  $p\,|\,n'$ by the expression (\ref{eq8}). 
Consequently we have  $p\,|\,\Delta$, that is, the distance $\Delta$ of the slopes represented by $m$ and $h^{-1}(m')$ satisfies $p\,|\,\Delta$. 

To prove the assertion (ii), we assume that $K$ represents a generator of $H_1(M)$. 
Let $p_0$ be a prime factor of $p$. 
Since we have the equations (\ref{eq2}) and (\ref{eq3}) in Lemma~\ref{lem}, for $p_0$, we also have the following. 
\begin{eqnarray}
\label{eq9}
 n \alpha + n'w \beta &\equiv& 0\ \  (\bmod\  p_0), \\
\label{eq10} 
qw \alpha &\equiv& 0\ \  (\bmod\  p_0). 
\end{eqnarray}
If $\alpha \not\equiv 0\ (\bmod\  p_0)$, then the equation (\ref{eq10}) implies $w \equiv 0\ (\bmod\  p_0)$ because $p$ and $q$ are relatively prime. 
This contradicts the equation (\ref{eq4}) in Lemma~\ref{lem}. 
Thus we see that $\alpha \equiv 0\ (\bmod\  p_0)$ and hence $\beta \not\equiv 0\ (\bmod\  p_0)$ by $(\alpha, \beta) =1$ if $\alpha \ne 0$ or by $(p,\beta)=1$ if $\alpha=0$ by Lemma~\ref{lem}. 
Then we have  $n'w \equiv 0\ (\bmod\  p_0)$ by the equation (\ref{eq9}). 
Hence by the equation (\ref{eq4}) in Lemma~\ref{lem}, we see that 
\begin{eqnarray}
\label{eq11}
p_0\,|\,n' 
\end{eqnarray}
for any prime factor $p_0$ of $p$. 
\end{proof}

\begin{proof}[Proof of Theorem~\ref{MainThm}]
Let $M$ be a non-hyperbolic homology lens space with $H_1(M) \cong \mathbb{Z}_p$ for a prime $p>8$. 
Suppose that a pair of hyperbolic, not null-homologous knots in $M$ have homeomorphic complements. 
Let $h$ be the homeomorphism between the exteriors of the knots. 
Then, by Theorem~\ref{Temp} (i), for the meridians $m$, $m'$ of the knots, the distance $\Delta$ of the slopes represented by $m$ and $h^{-1}(m')$ satisfies $p\,|\,\Delta$. 
That is, the distance is a multiple of $p>8$. 
Assume for the contrary that the two knots are inequivalent. 
Then the distance is not zero, and so, it has to be greater than 8. 
However, since $M$ is non-hyperbolic and the knots are hyperbolic, the distance must be at most 8 by the result of Lackenby and Meyerhoff \cite{LackenbyMeyerhoff}. 
A contradiction occurs. 
\end{proof}

\begin{proof}[Proof of Theorem~\ref{MainThm2}]
Let $M$ be a homology lens space with $H_1(M)  \cong \mathbb{Z}_p$ for a positive integer $p$, 
and let $K$ be a knot in $M$ representing a generator of $H_1(M)$. 
Suppose that there exists a knot $K'$ in $M$ which are not equivalent to $K$ but $E(K';M) \cong E(K;M)$. 
Let $h$ be the homeomorphism between the exteriors. 
Then, by Theorem~\ref{Temp} (ii), for the meridians $m$, $m'$ of the knots, the distance $\Delta$ of the slopes represented by $m$ and $h^{-1}(m')$ satisfies $p_0\,|\,\Delta $ for any prime factor $p_0$ of $p$. 

To prove $(1)$ of Theorem~\ref{MainThm2}, suppose that $M$ is a lens space. 
It is known that $(1)$ of Theorem~\ref{MainThm2} has been proven for knots with Seifert fibered exteriors in \cite[Theorem 1]{Rong}. 
It also follows from \cite[Theorem 8.2]{Gainullin} that $(1)$ of Theorem~\ref{MainThm2}  holds for null-homologous knots. 
Hence we may assume that $K$ is not null-homologous and its exterior is not a Seifert fibered space. 
Then $\Delta = 1$ by the Cyclic Surgery Theorem \cite{CGLS}. This contradicts the expression (\ref{eq11}).  

We next show $(2)$ of Theorem~\ref{MainThm2}. 
Suppose that  $M$ is spherical and $K$ is hyperbolic. 
Then it follows from \cite[Theorem 1.1(1)]{BoyerZhang} that the distance $\Delta$ is at most five, i.e., $\Delta \le 5$.  
By the expression (\ref{eq11}), we see that prime factors of $p$ are  $2$, $3$ or $5$. 

We similarly obtain $(3)$ of Theorem~\ref{MainThm2}.  
Suppose that  $M$ is non-hyperbolic and $K$ is hyperbolic. 
Then it follows from \cite{LackenbyMeyerhoff} that the distance $\Delta$ is at most eight, i.e., $\Delta \le 8$.  
By the expression (\ref{eq11}), we see that prime factors of $p$ are  $2$, $3$, $5$ or $7$. 
\end{proof}

\section{$L(4,q)$}

In this section, we give a proof of Theorem~\ref{L4q}. 

\begin{proof}[Proof of Theorem~\ref{L4q}]
Let $K_0$ be the trivial knot in the 3-sphere $S^3$. 
As in the previous section, we take the meridian-preferred longitude system $(\mu, \lambda)$ on $\partial E(K_0;S^3)$. 
Then a lens space $L(p,q)$ is obtained by the $(-p/q)$-surgery on $K_0$. 

Let $K$ be a knot in $L(p,q)$. It is known that Theorem \ref{L4q} has been proven for knots with Seifert fibered exteriors in \cite[Theorem 1]{Rong}. 
It also follows from \cite[Theorem 8.2]{Gainullin} that Theorem \ref{L4q}  holds for null-homologous knots. 
Hence we may assume that $K$ is not null-homologous and its exterior is not a Seifert fibered space. 

Let $K'$ be a not null-homologous knot in $L(p,q)$ with the exterior $E(K';L(p,q))$. 
Suppose that there exists a homeomorphism $h$ from the exterior $E(K;L(p,q))$ of $K$ to $E(K';M)$. 
Let $m'$ be a meridian of $K'$, and take the slope represented by $h^{-1}(m')$ on $\partial E(K;M)$, and assume that it corresponds to $n/n' \in \mathbb{Q} \cup \{ 1/0 \}$. 
Since we assume that $E(K;L(p,q))$ is not Seifert fibered, the surgery has to be integral, i.e., we may assume that $n'=1$ by the Cyclic Surgery Theorem \cite{CGLS}. 
Hence we see that the $n$-surgery on $K$ in $L(p,q)$ yields $L(p,q)$ again. 
That is, $M_K(n/1) \cong L(p,q)$, in particular, $H_1 (M_K(n/1)) \cong H_1 (L(p,q)) \cong \mathbb{Z}_p$ holds. 

Then, by Lemma~\ref{lem}, we have the presentation of $H_1 (M_K(n/1)) $ as follows. 
\[
H_1(M_K(n/1)) = \langle [m], [\mu] \ | \ n[m] + w [\mu] =0, \ - qw[m] + p [\mu] =0 \rangle.
\]
This implies that $H_1(L(p,q))$ has a presentation matrix 
\[
\begin{pmatrix}
n & w \\
-qw & p \\
\end{pmatrix}.
\]
Since the order of $H_1(L(p,q))$ is the modulus of the determinant of the matrix, 
we have $\pm p = np + w^2q$, i.e., 
\begin{eqnarray}
\label{eq1}
n=-\frac{w^2q}{p} \pm 1. 
\end{eqnarray}

Since we assume that $M_K(n/1)$ is homeomorphic to $L(p,q)$ and $H_1(L(p,q)) \cong \mathbb{Z}_p$,  
we also have the following as in the previous section: 
\begin{eqnarray}
\label{eq5}
n \alpha + w \beta &\equiv& 0\ \  (\bmod\  p), \\
\label{eq6} 
qw \alpha &\equiv& 0\ \  (\bmod\  p), \\
\label{eq7}
p\!\!\!\!  &\not|& \!\!\!\! w. 
\end{eqnarray}

Under the condition above, we further assume $p=4$. 
If $w$ is an odd integer, then $n$ is not an integer by the equation (\ref{eq1}) since $q$ is also an odd integer. 
Hence $w$ is an even integer. 
In particular, $w \equiv 2 \ (\bmod\  4)$ by the expression (\ref{eq7}). 
This implies that $w^2/p$ is an odd integer and hence we see that $n$ is an even integer by the equation (\ref{eq1}). 

On the other hand, $\alpha$ is an even integer by the equation (\ref{eq6}) since $w \equiv 2 \ (\bmod\  4)$ and $(p,q)=1$. 
If $\alpha = 0$, then $w \beta \equiv 0 \ (\bmod\  4)$ by the equation (\ref{eq5}). This implies that $\beta$ is an even integer 
and hence contradicts $(\alpha, \beta)=1$. Therefore $\alpha = 2$ and $\beta$ is an odd integer. 
Then by the equation (\ref{eq5}), 
\[
n \alpha + w \beta \equiv 2n + 2\beta \equiv 0\ \  (\bmod\  4). 
\]
Since $\beta$ is an odd integer, we see that $2n \equiv 2 \ (\bmod\  4)$, i.e., $n$ has to be an odd integer.  
A contradiction occurs. 
\end{proof}

\section*{Acknowledgements}
The authors would like to thank Tetsuya Ito for useful discussions. 
Ichihara is partially supported by JSPS KAKENHI Grant Number 18K03287.

\end{document}